\documentclass[a4paper,12pt]{amsart}

\usepackage{graphicx}
\usepackage{amsmath,amssymb,amsthm,amsfonts}
\usepackage{amssymb}

\newtheorem{thm}{Theorem}[section]

\newtheorem{lem}[thm]{Lemma}
\newtheorem{prop}[thm]{Proposition}
\newtheorem{rem}[thm]{Remark}

\newcommand{\bremark}{\begin{rem} \textup}
\newcommand{\eremark}{\end{rem} }

\newcommand{\cuad}{{\sqcap\kern-.68em\sqcup}}

\renewcommand{\rho}{\varrho}
\renewcommand{\theta}{\vartheta}

\begin{document}


\parindent 0pc
\parskip 6pt
\overfullrule=0pt

\title[Symmetry of solutions]{Symmetry of solutions of some semilinear elliptic equations with singular nonlinearities}

\author {A. Canino*, M. Grandinetti* and B. Sciunzi*}

\date{\today}

\address{* Dipartimento di Matematica, UNICAL,
Ponte Pietro  Bucci 31B, 87036 Arcavacata di Rende, Cosenza, Italy.}

\email{canino@mat.unical.it, grandinetti@mat.unical.it, sciunzi@mat.unical.it}




\begin{abstract}
We consider positive solutions to the singular semilinear elliptic equation $-\Delta\,u\,=\,\frac{1}{u^\gamma}\,+\,f(u)$, in bounded smooth domains, with zero Dirichlet boundary conditions.\\
\noindent We provide some weak and strong maximum principles for the $H^1_0(\Omega)$ part of the solution (the solution $u$  does generally not belong to $H^1_0(\Omega)$), that allow to deduce symmetry and monotonicity properties of the solutions, via the \emph{Moving Plane Method}.
\end{abstract}

\maketitle

\date{\today}



\maketitle

\section{introduction}

In this paper we study symmetry and monotonicity properties of the solutions to the problem
\begin{equation}
	\label{problem}
\begin{cases}
-\Delta\,u\,=\,\frac{1}{u^\gamma}\,+\,f(u) & \text{in $\Omega$,}  \\
u> 0 & \text{in $\Omega$,}  \\
u=0 & \text{on $\partial\Omega$.}
\end{cases}
\end{equation}
where $\gamma>0$, $\Omega $ is a bounded smooth domain and
$u\in\,C(\overline{\Omega})\cap C^2(\Omega)$. \\
\noindent Our main results will be proved under the following assumption
\begin{itemize}
\item[($H_p$)] $f$ is locally Lipschitz continuous, non-decreasing, $f(s)>0$ for $s>0$ and $f(0)\geq 0$.
\end{itemize}
As a model problem we may consider solutions to $-\Delta\,u\,=\,\frac{1}{u^\gamma}\,+\,u^q$ with $q>0$.\\
\noindent Since the pioneer results in \cite{crandall} and \cite{stuart}, singular semilinear elliptic equations have been considered by several authors. We refer to \cite{boccardo2,boccardo,Cancan,CanDeg,gras,GOW,saccon,saccon2,lair,lazer,marco}.\\
 The variational characterization of problem \eqref{problem} is not trivial. In fact, already in the case $f\equiv 0$, the condition $\gamma<3$ is necessary to have solutions in $H^1_0(\Omega)$ and  to have the associated energy functional $I \neq +\infty$, see  \cite{lazer}. A first attempt in this direction can be found in \cite{saccon} in the case $\gamma\leq 1$. \\
\noindent Later in \cite{CanDeg} a general approach was developed for any $\gamma>0$. The main idea in \cite{CanDeg}, that will be a key ingredient in the present paper, is a translation
of the energy functional and of the functions space used, based on the decomposition of the solutions of \eqref {problem}\,as
\begin{equation}\label{decomposition}
u=\, u_0\,+\,w
\end{equation}
where $w\in H^1_0(\Omega)$ and  $u_0\in\,C(\overline{\Omega})\cap C^2(\Omega)$ is the solution to the problem:
\begin{equation}
	\label{problem0}
\begin{cases}
-\Delta\,u_0\,=\,\frac{1}{{u_0}^\gamma}\, & \text{in $\Omega$,}  \\
u_0> 0 & \text{in $\Omega$,}  \\
u_0=0 & \text{on $\partial\Omega$.}
\end{cases}
\end{equation}

\noindent The solution $u_0$  is unique (see Lemma 2.8 in \cite{CanDeg}) and can be found via a sub-super solution method like in \cite{CanDeg} or via a truncation argument as in \cite{boccardo}. It follows by the comparison argument used in the proof of \cite{CanDeg} that the solution $u_0$ is continuous up to the boundary and is bounded away from zero in the interior of $\Omega$. This latter information also follows by \cite{boccardo} where the solution $u_0$ is obtained as the limit of an increasing sequence of positive solutions to a regularized problem.

\noindent The equation $-\Delta\,u_0\,=\,\frac{1}{{u_0}^\gamma}$  consequently can be understood in the weak distributional sense with test functions with compact support in $\Omega$, that is
\begin{equation}\label{eq:fgakghfhksajdgfja}
\int_\Omega\,(Du_0,  D\varphi)\,dx  = \int_\Omega \frac{\varphi}{{u_0}^\gamma}\,dx  \qquad\forall \varphi\in C^1_c(\Omega).
\end{equation}
Actually, the solution is fulfilled in the classical sense in the interior of $\Omega$ by standard regularity results,
since $u_0$ is strictly positive in the interior of the domain.\\

\noindent In any case, taking into account \cite{lazer}, for $\gamma\geq 3$ $u_0$ does not belong to $H^1_0(\Omega)$, hence $u$ does not belong to $H^1_0(\Omega)$ too.

\noindent The proof of our symmetry result is based on the well known \emph{Moving Plane Method}~(see \cite{serrin}),
that was used in a clever way in the celebrated paper \cite{GNN} in the semilinear nondegenerate case. Actually our proof is more similar to the one of \cite{BN} and is based on the weak comparison principle in small domains.\\

\noindent Because of the singular nature of our problem, we have to take care of two difficulties, namely:
\begin{itemize}
\item[-]  $u$ does not belong to $H^1_0(\Omega)$,
\item[-]
$\frac{1}{s^\gamma}\,+\,f(s)$ is not  Lipschitz continuous at zero.
\end{itemize}

\noindent

 This causes that a straightforward modification of the moving plane technique is not possible in our setting and for this reason we need a new technique based on the decomposition in \eqref{decomposition}.

\noindent Let us state our symmetry result:

\begin{thm}\label{t3}
Let $u\in\,C(\overline{\Omega})\cap C^2(\Omega)$ be  a solution to \eqref{problem} with $f$ satisfying $(H_p)$.  Assume that the domain $\Omega$ is  strictly convex w.r.t. the $\nu-$direction $(\nu \in S^{N-1})$ and symmetric w.r.t. $T_0^\nu$, where
$$ T_0^\nu=\{x\in \mathbb{R}^N : x\cdot \nu=0\}.$$
Then $u$ is symmetric w.r.t. $T_0^\nu$ and non-decreasing w.r.t. the $\nu-$direction in~$\Omega_0^\nu$, where
\begin{equation}\nonumber
\Omega_0^\nu=\{x\in \Omega:x\cdot \nu <0\}\,.
\end{equation}
 Moreover, if  $\Omega$ is a ball, then $u$ is radially symmetric with $\frac{\partial u}{\partial r}(r)<0$ for $r\neq 0$.
\end{thm}

\noindent For the reader's convenience, we describe here below the scheme of the proof.

\begin{itemize}
\item [(i)]   Since, by \cite{boccardo}, $u_0$ is the limit of a sequence $u_n$ of solutions to a regularized problem \eqref{problem0T}, we deduce  symmetry and monotonicity properties of $u_n$, and consequently of $u_0$, applying the moving plane procedure  in a standard way to the regularized problem \eqref{problem0T}.

    \item [(ii)] By (i),  recalling the decomposition in \eqref {decomposition}\,: $u=u_0+w$, we are reduced to prove symmetry and monotonicity properties of $w$. To do this, in Section \ref{principi}, we prove some comparison principles for $w$ needed in the application of the moving plane procedure.

        \item [(iii)] In Section \ref{simmetria}, we carry out the adaptation
        of the moving plane procedure to the study of the monotonicity and symmetry of $w$.
         It is worth emphasizing that the moving plane procedure is applied in our approach only to the $H^1_0(\Omega)$ part of $u$.\\
         \noindent Note also that Theorem \ref{t3} is proved in Section \ref{simmetria2} exploiting the more general result Proposition \ref{prosimmetri2}.
\end{itemize}

\section{Notations}

To state the next results we need some notations. Let $\nu$ be a direction in $\mathbb{R}^N$ with $|\nu|=1$. Given a real number $\lambda$ we set
\begin{equation}\nonumber
T_\lambda^\nu=\{x\in \mathbb{R}^N:x\cdot\nu=\lambda\},
\end{equation}
\begin{equation}\nonumber
\Omega_\lambda^\nu=\{x\in \Omega:x\cdot \nu <\lambda\}
\end{equation}
and
\begin{equation}\nonumber
x_\lambda^\nu=R_\lambda^\nu(x)=x+2(\lambda -x\cdot\nu)\nu,
\end{equation}
that is the reflection trough the hyperplane $T_\lambda^\nu$. Moreover we set
\begin{equation}\nonumber
(\Omega_\lambda^\nu)'=R_\lambda^\nu(\Omega_\lambda^\nu).
\end{equation}
Observe that $(\Omega_\lambda^\nu)'$ may be not contained in $\Omega$. Also we take
\begin{equation}\nonumber
a(\nu)=\inf _{x\in\Omega}x\cdot \nu.
\end{equation}
When $\lambda >a(\nu)$, since $\Omega_\lambda^\nu$ is nonempty, we set
\begin{equation}\nonumber
\Lambda_1(\nu)=\{\lambda : (\Omega_t^\nu)'\subset \Omega\,\, \text{for any}\,\,a(\nu)< t\leq\lambda\},\end{equation}
and
\begin{equation}\nonumber
\lambda_1(\nu)=\sup \Lambda_1(\nu).
\end{equation}
Moreover we set
\begin{equation}\nonumber
u_\lambda^\nu(x)=u(x_\lambda^\nu)\,,
\end{equation}
 for any $a(\nu)<\lambda\leq \lambda_1(\nu)$.\\
 \noindent Recalling the decomposition  of the solutions of \eqref {problem} (see \eqref{decomposition}) as
\begin{equation}\nonumber
u=\, u_0\,+\,w,
\end{equation}
we set
\begin{equation}\nonumber
{u_0}_\lambda^\nu(x)={u_0}(x_\lambda^\nu)\,,
\end{equation}
and
\begin{equation}\nonumber
w_\lambda^\nu(x)=w(x_\lambda^\nu)\,.
\end{equation}
\section{Symmetry properties of $u_0$}\label{gg}
Basing on the construction of the solution $u_0$ of \eqref{problem0} we prove in this section some useful symmetry and monotonicity results for $u_0$.

\begin{prop}\label{prosimmetri2000}
Let $u_0\in\,C(\overline{\Omega})\cap C^2(\Omega)$ be the solution to \eqref{problem0}.
Then, for any $$a(\nu)<\lambda<\lambda_1(\nu)$$ we have
\begin{equation}\label{zxclu:1.8222000}
u_0(x)< {u_0}_{\lambda}^\nu(x),\qquad \forall x \in \Omega _{\lambda}^\nu\,
\end{equation}
and
\begin{equation}\label{zxclu:1.19222u000}
\frac{\partial u_0}{\partial \nu}(x) > 0, \qquad \forall x \in
\Omega _{\lambda_1(\nu)}^\nu.
\end{equation}
\end{prop}
\begin{proof}
Let $u_n\in H^1_0(\Omega)\cap C(\overline\Omega)$ be the unique solution to
\begin{equation}
	\label{problem0T}
\begin{cases}
-\Delta\,u_n\,=\,\frac{1}{(u_n+\frac{1}{n})^\gamma}\, & \text{for $x\in\Omega$,}  \\
u_n> 0 & \text{for $x\in\Omega$,}  \\
u_n=0 & \text{for $x\in\partial\Omega$.}
\end{cases}
\end{equation}
The existence of $u_n$ was proved in \cite{boccardo} and the uniqueness follows by \cite{CanDeg}. Since the problem is no more singular, by standard
elliptic estimates it follows that $u_n\in C^2(\overline\Omega)$. Therefore we can use the moving plane technique exactly as in \cite{BN,GNN} to deduce that the statement of our proposition holds true for each $u_n$.
By \cite{boccardo} $u_n$ converges to $u_0$ a.e. as $n$ tends to infinity and therefore \eqref{zxclu:1.8222000} follows passing to the limit.
Finally in the same way
\[
\frac{\partial u_0}{\partial \nu}(x) \geq 0, \qquad \forall x \in
\Omega _{\lambda_1(\nu)}^\nu\,,
\]
and therefore \eqref{zxclu:1.19222u000} follows via the strong maximum principle.
\end{proof}
As a consequence of Proposition \ref{prosimmetri2000}, we get
\begin{prop}\label{t3000}
Let $u_0\in\,C(\overline{\Omega})\cap C^2(\Omega)$ be the solution of \eqref{problem0} and assume that the domain $\Omega$  is strictly convex w.r.t. the $\nu-$direction $(\nu \in S^{N-1})$ and symmetric w.r.t. $T_0^\nu$.
Then $u_0$ is symmetric w.r.t. $T_0^\nu$ and non-decreasing w.r.t. the $\nu-$direction in~$\Omega_0^\nu$.
 Moreover, if  $\Omega$ is a ball, then $u_0$ is radially symmetric with $\frac{\partial u_0}{\partial r}(r)<0$ for $r\neq 0$.
\end{prop}

\section{comparison principles}\label{principi}

Let us start with the following
\begin{lem}\label{lemmaappendix}
Let $\gamma >0$. Consider the function
\begin{equation}\nonumber
g_\gamma(x,y,z,h)\,:=\,x^\gamma(x+y)^\gamma(z+h)^\gamma\,+\,x^\gamma z^\gamma(z+h)^\gamma
\,-\,z^\gamma(x+y)^\gamma(z+h)^\gamma\,-\,x^\gamma z^\gamma(x+y)^\gamma
\end{equation}
and the domain $D\subset \mathbb{R}^4$ defined by
\begin{equation}\nonumber
D\,:=\, \Big\{(x,y,z,h)\,|\, 0\leq x\leq z\, ;\, 0\leq h\leq y \Big\}\,.
\end{equation}
Then  it follows that $g_\gamma \leq 0$ in $D$.
\end{lem}
\begin{proof}
Since $x\leq z$, by direct calculation we get
\begin{equation}\nonumber
\begin{split}
\frac{\partial g_\gamma}{\partial y}(x,y,z,h)\,=\,
 \gamma x^\gamma(x+y)^{\gamma-1}(z+h)^\gamma \,-\,
\gamma z^\gamma(x+y)^{\gamma-1}(z+h)^\gamma\,-\,\gamma x^\gamma z^\gamma(x+y)^{\gamma-1}\,
\leq 0
\end{split}
\end{equation}
Therefore we are reduced to prove that  $g_\gamma \leq 0$ in $D\cap\{h=y\}$, that is
\[
g_\gamma(x,y,z,y)= x^\gamma(x+y)^\gamma(z+y)^\gamma\,+\,x^\gamma z^\gamma(z+y)^\gamma
\,-\,z^\gamma(x+y)^\gamma(z+y)^\gamma\,-\,x^\gamma z^\gamma(x+y)^\gamma\,
\leq 0\,.
\]
For $x=0$ the thesis follows immediately. For $x>0$ we note that
\[
g_\gamma(x,y,z,y)\,=\,-\Big(\frac{1}{x^\gamma}-\frac{1}{z^\gamma}\,+\,  \frac{1}{(z+y)^\gamma}-\frac{1}{(x+y)^\gamma}\Big)
(x^\gamma z^\gamma  (z+y)^\gamma(x+y)^\gamma)
\]
and the conclusion follows exploiting the fact that, for $0<x\leq z$ fixed, the function
$$
\tilde g_\gamma(t)\,:=\,x^{-\gamma}-z^{-\gamma}+(z+t)^{-\gamma}-(x+t)^{-\gamma}
$$
is increasing in $[0\,,\,\infty)$ and $\tilde g_\gamma(0)=0$.
\end{proof}

\begin{lem}\label{positivity}
Let $u\in\,C(\overline{\Omega})\cap C^2(\Omega)$ be a solution to problem \eqref{problem} with $\gamma>0$. Assume that $\Omega $ is a bounded smooth domain and that $f$ is locally Lipschitz continuous, $f(s)>0$ for $s>0$ and $f(0)\geq 0$.  Let $w$ be given by \eqref{decomposition}. \\
\noindent Then it follows
\[
w>0\qquad\text{in}\qquad \Omega\,.
\]
\end{lem}

 \begin{proof}
Since $u\in  C(\overline \Omega)\cap C^2 (\Omega)$ and
 $u_0\in  C(\overline \Omega)\cap C^2 (\Omega)\,$, then $w\in H^1_0(\Omega)\cap C(\overline \Omega)\cap C^2 (\Omega)\,$.\\
\noindent
By hypothesis on $f$, it follows that $u$ is a super-solution (following Definition 2.5 in \cite{CanDeg}) to the equation
\[
-\Delta v=\frac{1}{v^\gamma}\,.
\]
Therefore, by Lemma 2.8 in \cite{CanDeg} we get that
\[
u\geq u_0\quad\text{in}\,\,\Omega \qquad \text{and therefore}\qquad w\geq 0\quad\text{in}\,\,\Omega\,.
\]
Now let us  show that $w>0$ in the interior of $\Omega$ via the maximum principle exploited in regions where the problem is not singular. More precisely let us assume by contradiction that there exists a point $x_0\in\Omega$ such that $w(x_0)=0$ and let $r=r(x_0)>0$ such that $B_r(x_0)\subset\subset \Omega$. We have, in the classical sense, in $B_r(x_0)$
\[
-\Delta w\,=\,-\Delta u\,+\,\Delta u_0\,=\, \frac{1}{(u_0+w)^\gamma}\,+\,f(u)\,-\,\frac{1}{u_0^\gamma}\geq \, \frac{1}{(u_0+w)^\gamma}-\frac{1}{u_0^\gamma}\,.
\]
Since $u_0(x_0)>0$ we can assume that $u_0$ is positive in $B_r(x_0)$. Therefore we get that
\[
\frac{1}{(u_0+w)^\gamma}-\frac{1}{u_0^\gamma}=c(x)\,(u_0+w-u_0)=c(x)\,w
\]
for some bounded coefficient $c(x)$. Thus there exists $\Lambda>0$ such that
\noindent
$\frac{1}{(u_0+w)^\gamma}-\frac{1}{u_0^\gamma}\,+\,\Lambda \,w\geq 0$ in $B_r(x_0)$, so that
\[
-\Delta w\,+\,\Lambda \,w\geq 0\qquad \text{in}\qquad B_r(x_0)\,.
\]
By the strong maximum principle we  get $w\equiv 0$ in  $B_r(x_0)$ and by a covering argument that $w\equiv 0$ in
$\Omega$. But $w\equiv 0$ in
$\Omega$ implies $f=0$ and we get a contradiction.

\end{proof}

\begin{prop}[A strong maximum principle]\label{pr:wcpmax}
Let $a(\nu)<\lambda<\lambda_1(\nu)$ and $\Omega '$  a sub-domain of $ \Omega_\lambda^\nu$.
Assume that $u\in\,C(\overline{\Omega})\cap C^2(\Omega)$ is  a solution to \eqref{problem} with $f$
satisfying $(H_p)$.\\
\noindent Let $w$ be given by \eqref{decomposition} and assume that
\[
\frac{\partial \,w}{\partial \nu}\geq 0\qquad \text{in}\qquad \Omega '\,.
\]
Then it holds the alternative
\[
\frac{\partial \,w}{\partial \nu}> 0\quad \text{in}\quad  \Omega '\qquad \text{or}\qquad
\frac{\partial \,w}{\partial \nu} = 0\quad \text{in}\quad \Omega'\,.
\]
\end{prop}
\begin{proof}
Let us use the short hand notation
$w_\nu\,:=\,\frac{\partial w}{\partial \nu} $ and ${u_0}_\nu\,:=\,\frac{\partial u_0}{\partial \nu} $. Since $f'\geq 0$ a.e.\footnote{Note that, even if $f'$ exist a.e., the term $f'(u)(w_\nu+{u_0}_\nu)$ makes sense in the weak Sobolev meaning thanks to Stampacchia's Theorem. } by assumption ($H_p$), ${u_0}_\nu\geq 0$ in $\Omega'$ by Proposition \ref{prosimmetri2000},  $u\geq u_0$ by Lemma \ref{positivity} and finally $w_\nu\geq 0$ in $\Omega'$ by assumption, differentiating the equation in \eqref{problem} we get that $w_\nu$ solves
 \begin{equation}\nonumber
 \begin{split}
 &-\Delta w_\nu=-\frac{\gamma}{u^{\gamma+1}}w_\nu\,+\,
 f'(u)(w_\nu+{u_0}_\nu)\,+\,\gamma\,\Big(\frac{1}{u_0^{\gamma+1}}-\frac{1}{u^{\gamma+1}}\Big){u_0}_\nu\\
 &\geq-\frac{\gamma}{u^{\gamma+1}}w_\nu\,,
 \end{split}
 \end{equation}

 \noindent We recall now that $u$ is bounded away from zero in $\Omega'$, and therefore we find $\Lambda >0$ such that
  \begin{equation}\nonumber
 \begin{split}
 &-\Delta w_\nu \geq-\frac{\gamma}{u^{\gamma+1}}w_\nu\geq - \Lambda\,w_\nu\,,
 \end{split}
 \end{equation}
 so that the conclusion follows by the standard strong maximum principle \cite{GT}.

\end{proof}

\begin{prop}[Weak Comparison Principle in small domains]\label{pr:wcp}
Let $a(\nu)<\lambda<\lambda_1(\nu)$ and $\Omega ' \subseteq \Omega_\lambda^\nu$. Assume that $u\in\,C(\overline{\Omega})\cap C^2(\Omega)$ is  a solution to \eqref{problem} with $f$
satisfying $(H_p)$.
Let $w$ be given by \eqref{decomposition} and assume that
\[
w\leq w_\lambda^\nu\qquad \text{on}\qquad \partial \Omega '\,.
\]

Then there exists a positive constant $\delta=\delta\left (u\,,\,f\right)$ such that, if
$ \mathcal{L}( \Omega ') \leq \delta,$ then
$$ w\leq  w_\lambda^\nu \quad \text{in }  \Omega '.$$
\end{prop}
\begin{proof}
We have
\begin{eqnarray}\label{eq:wcp1}
-\Delta (u_0+w) \,\,\,=\frac{1}{(u_0+w)^\gamma}\,+\,f(u_0+w)\qquad \qquad\text{in }\Omega,\\\label{eq:wcp2}
-\Delta ({u_0}_\lambda^\nu+w_\lambda^\nu) =\frac{1}{({u_0}_\lambda^\nu+w_\lambda^\nu)^\gamma}\,+\,f({u_0}_\lambda^\nu+w_\lambda^\nu) \quad \,\,\,\,\text{in }\Omega,
\end{eqnarray}

Since $(w-w_\lambda^\nu)^+\in H^1_0(\Omega')$ we can consider a sequence of positive functions $\psi_n$ such that
\[
\psi_n\in C_c^\infty(\Omega')\qquad\text{and}\qquad \psi_n\overset{H^1_0(\Omega')}{\longrightarrow}(w-w_\lambda^\nu)^+\,.
\]
We  can also assume that $supp\,\, \psi_n\subseteq supp\,\, (w-w_\lambda^\nu)^+$.
We plug $\psi_n$ into the weak formulation of \eqref{eq:wcp1} and \eqref{eq:wcp2} and subtracting we get

\begin{eqnarray}\label{eq:smm11}
&&\int_{\Omega'}(D(u_0+w)-D({u_0}_\lambda^\nu+w_\lambda^\nu)\,,\,D\psi_n)\,dx\\\nonumber &=&
\int_{\Omega'} \big(\frac{1}{(u_0+w)^\gamma}\,+\,f(u_0+w)-\frac{1}{({u_0}_\lambda^\nu+w_\lambda^\nu)^\gamma}\,-
\,f({u_0}_\lambda^\nu+w_\lambda^\nu)\big)\psi_n\,dx\,.
\end{eqnarray}
Recalling that $u_0$ and ${u_0}_\lambda^\nu$ solve \eqref{problem0} we deduce

\begin{eqnarray}\label{eq:smm112b}
&&\int_{\Omega'}(D(w-w_\lambda^\nu)\,,\,D\psi_n)\,dx=
\int_{\Omega'} \Big(\frac{1}{({u_0}_\lambda^\nu)^\gamma}-\frac{1}{(u_0)^\gamma}
+\frac{1}{(u_0+w)^\gamma}-
\frac{1}{({u_0}_\lambda^\nu+w_\lambda^\nu)^\gamma}
\Big)\psi_n dx\\\nonumber &+&
\int_{\Omega'}\Big( f(u_0+w)-f({u_0}_\lambda^\nu+w_\lambda^\nu)\big)\psi_n\,dx\,.
\end{eqnarray}
Since $u_0\leq {u_0}_\lambda^\nu$ in $\Omega_\lambda^\nu$ and $w\geq w_\lambda^\nu$ on the support of $\psi_n$, by applying Lemma \ref{lemmaappendix} with $u_0=x$, $w=y$, ${u_0}_\lambda^\nu=z$ and ${w}_\lambda^\nu=h$   we get

\begin{equation}\nonumber
\begin{split}
&(u_0)^\gamma(u_0+w)^\gamma({u_0}_\lambda^\nu+w_\lambda^\nu)^\gamma+
(u_0)^\gamma({u_0}_\lambda^\nu)^\gamma({u_0}_\lambda^\nu+w_\lambda^\nu)^\gamma\\
&-({u_0}_\lambda^\nu)^\gamma(u_0+w)^\gamma({u_0}_\lambda^\nu+w_\lambda^\nu)^\gamma-
(u_0)^\gamma({u_0}_\lambda^\nu)^\gamma(u_0+w)^\gamma
\leq 0
\end{split}
\end{equation}
and then $\Big(\frac{1}{({u_0}_\lambda^\nu)^\gamma}-\frac{1}{(u_0)^\gamma}
+\frac{1}{(u_0+w)^\gamma}-
\frac{1}{({u_0}_\lambda^\nu+w_\lambda^\nu)^\gamma}
\Big)\leq 0$.\\
Therefore, by assumption $(H_p)$, we find a constant $C>0$ such that
\begin{eqnarray}\label{eq:smm112btre}
&&\int_{\Omega'}(D(w-w_\lambda^\nu)\,,\,D\psi_n)\,dx
\leq
\int_{\Omega'}\Big( f(u_0+w)-f({u_0}_\lambda^\nu+w_\lambda^\nu)\big)\psi_n\,dx\,\\\nonumber &\leq& \int_{\Omega'}\Big( f({u_0}_\lambda^\nu+w)-f({u_0}_\lambda^\nu+w_\lambda^\nu)\big)\psi_n\,dx \leq C\int_{\Omega'}(w-w_\lambda^\nu)\psi_n\,dx\,.
\end{eqnarray}
We now pass to the limit for $n\rightarrow \infty$,\,we get
\begin{equation}\nonumber
\int_{\Omega'}|D(w-w_\lambda^\nu)^+|^2\,dx\leq C\int_{\Omega'}|(w-w_\lambda^\nu)^+|^2\,dx
\end{equation}
and by Poincar\'{e} inequality
\begin{equation}\nonumber
\int_{\Omega'}|D(w-w_\lambda^\nu)^+|^2\,dx\leq C\,C_p(\Omega')\int_{\Omega'}|D(w-w_\lambda^\nu)^+|^2\,dx\,.
\end{equation}
For $\delta$ small it follows that $C\,C_p(\Omega')<1$ which shows that actually $(w-w_\lambda^\nu)^+=0$ and the thesis follows.

\end{proof}

\begin{lem}[Strong Comparison Principle]\label{positivitycomp}
Let $u\in\,C(\overline{\Omega})\cap C^2(\Omega)$ be a solution to problem \eqref{problem}, with $f$ satisfying $(H_p)$. Let $w$ be given by \eqref{decomposition} and assume that,
 for some $a(\nu)<\lambda\leq \lambda_1(\Omega)$, we have
\[
w\leq w_\lambda^\nu \qquad \text{in}\qquad \Omega_\lambda^\nu\,.
\]
Then $w < w_\lambda^\nu$ in $\Omega_\lambda^\nu$ unless $w \equiv w_\lambda^\nu$ in $\Omega_\lambda^\nu$.
\end{lem}

 \begin{proof}
 Let us assume  that there exists a point $x_0\in\Omega_\lambda^\nu$ such that $w(x_0)=w_\lambda^\nu(x_0)$ and let $r=r(x_0)>0$ such that $B_r(x_0)\subset\subset \Omega_\lambda^\nu$. We have, in the classical sense, in $B_r(x_0)$
\begin{equation}\label{fkjebdbdbvkbvk}
\begin{split}
&-\Delta (w_\lambda^\nu-w)\,=\,-\Delta (u_\lambda^\nu-{u_0}_\lambda^\nu)\,+\,\Delta (u- u_0)
\\
&= \Big(\frac{1}{u_0^\gamma}-\frac{1}{({u_0}_\lambda^\nu)^\gamma}\,+\,  \frac{1}{({u_0}_\lambda^\nu+w)^\gamma}-\frac{1}{(u_0+w)^\gamma}  \Big)\,+\,\Big(f({u_0}_\lambda^\nu+w_\lambda^\nu)-f(u_0+w)\Big)\\
&+ \frac{1}{({u_0}_\lambda^\nu+w_\lambda^\nu)^\gamma}-\frac{1}{({u_0}_\lambda^\nu+w)^\gamma}\,.
\end{split}
\end{equation}
Since $f$ is non-decreasing by assumption, $u_0\leq {u_0}_\lambda^\nu$ in $\Omega_\lambda^\nu$ by Proposition \ref{prosimmetri2000} and  $w\leq w_\lambda^\nu$ in $\Omega_\lambda^\nu$,  we get
 $$f({u_0}_\lambda^\nu+w_\lambda^\nu)-f(u_0+w)\geq 0.$$
Moreover, since for $0<a\leq b$ the function
$
g(t)\,:=\,a^{-\gamma}-b^{-\gamma}+(b+t)^{-\gamma}-(a+t)^{-\gamma}\,
$
is increasing in $[0\,,\,\infty)$, we also have
\[
\Big(\frac{1}{u_0^\gamma}-\frac{1}{({u_0}_\lambda^\nu)^\gamma}\,+\,  \frac{1}{({u_0}_\lambda^\nu+w)^\gamma}-\frac{1}{(u_0+w)^\gamma}  \Big)\geq 0\,.
\]
and by \eqref{fkjebdbdbvkbvk}\, we get
\begin{equation}\nonumber
\begin{split}
-\Delta (w_\lambda^\nu-w)\geq
 \frac{1}{({u_0}_\lambda^\nu+w_\lambda^\nu)^\gamma}-\frac{1}{({u_0}_\lambda^\nu+w)^\gamma}\,.
\end{split}
\end{equation}
Since ${u_0}_\lambda^\nu(x_0)>0$, arguing as in Lemma \ref{positivity},  we  find  $\Lambda>0$ such that, eventually reducing $r$, it results
$\frac{1}{({u_0}_\lambda^\nu+w_\lambda^\nu)^\gamma}-\frac{1}{({u_0}_\lambda^\nu+w)^\gamma}+\Lambda \,(w_\lambda^\nu-w)\geq 0$ in $B_r(x_0)$, so that
\[
-\Delta (w_\lambda^\nu-w)\,+\,\Lambda \,(w_\lambda^\nu-w)\geq 0\qquad \text{in}\qquad B_r(x_0)\,.
\]
By the strong maximum principle \cite{GT} it follows  $w_\lambda^\nu-w\equiv 0$ in  $B_r(x_0)$, and by a covering argument $w_\lambda^\nu-w\equiv 0$ in
$\Omega_\lambda^\nu$, proving the result.
\end{proof}

\section{symmetry}\label{simmetria}
\begin{prop}\label{prosimmetri2}
Let $u\in\,C(\overline{\Omega})\cap C^2(\Omega)$ be a solution to \eqref{problem}. Let $w$ be given by \eqref{decomposition}.\\
\noindent Then, for any $$a(\nu)<\lambda<\lambda_1(\nu)$$ we have
\begin{equation}\label{zxclu:1.8222}
w(x)< w_{\lambda}^\nu(x),\qquad \forall x \in \Omega _{\lambda}^\nu.
\end{equation}
Moreover
\begin{equation}\label{zxclu:1.9222u}
\frac{\partial w}{\partial \nu}(x) > 0, \qquad \forall x \in
\Omega _{\lambda_1(\nu)}^\nu.
\end{equation}
Finally, \eqref{zxclu:1.8222} and \eqref{zxclu:1.9222u} hold true replacing $w$ by $u$.
\end{prop}

\begin{proof}
Let $\lambda > a(\nu)$. Since $w>0$ in $\Omega$ by Lemma \ref{positivity} we have:
\[
w\leq w_\lambda^\nu \qquad \text{on}\qquad \partial \Omega_\lambda^\nu\,.
\]
 Therefore, assuming that $\mathcal{L}(\Omega_\lambda^\nu)$ is sufficiently small (say for $\lambda-a(\nu)$ sufficiently small) so that Proposition \ref{pr:wcp} applies, we get
\begin{equation}\label{fjgfjf}
w\leq w_\lambda^\nu \qquad \text{in}\qquad  \Omega_\lambda^\nu\,,
\end{equation}
and actually $w< w_\lambda^\nu$ in $\Omega_\lambda^\nu$ by the Strong Comparison Principle
(Lemma \ref{positivitycomp}).\\
\noindent Let us  define
\begin{equation}\nonumber
\Lambda_0=\{\lambda > a(\nu): w\leq w_{t}^\nu\,\,\,\text{in}\,\,\,\Omega_t^\nu\,\,\,\text{for all
$t\in(a(\nu),\lambda]$}\}
 \end{equation}
 which is not empty thanks to \eqref{fjgfjf}. Also set
 \begin{equation}\nonumber
\lambda_0=\sup\,\Lambda_0.
\end{equation}
By the definition of $\lambda_1(\nu)$, to prove our result we have to show that actually  $\lambda_0 = \lambda_1(\nu)$.\\
\noindent
Assume otherwise that $\lambda_0 < \lambda_1(\nu)$ and note that, by continuity, we obtain $w\leq w_{\lambda_0}^\nu$ in $\Omega_{\lambda_0}^\nu$.
By the Strong Comparison Principle (Lemma \ref{positivitycomp}), it follows $w< w_{\lambda_0}^\nu$ in $\Omega_{\lambda_0}^\nu$
unless  $w= w_{\lambda_0}^\nu$ in $\Omega_{\lambda_0}^\nu$.
Because of  the zero Dirichlet boundary conditions and since $w>0$ in the interior of the domain,
 the case $w\equiv w_{\lambda_0}^\nu$ in $\Omega_{\lambda_0}^\nu$ is  not possible. Thus $w< w_{\lambda_0}^\nu$ in $\Omega_{\lambda_0}^\nu$.\\
We can now consider $\delta $ given by Proposition \ref{pr:wcp} so that the Weak Comparison Principle holds true in any sub-domain $\Omega'$ if $\mathcal{L}(\Omega')\leq \delta$. Fix a compact set $\mathcal{K} \subset\Omega_{\lambda_0}^\nu$ so that $\mathcal{L}(\Omega_{\lambda_0}^\nu\setminus \mathcal{K})\leq \frac{\delta}{2}$.
By compactness we find $\sigma >0$ such that
\[
w_{\lambda_0}^\nu-w\geq 2\sigma>0\,\quad \text{in}\quad \mathcal{K}\,.
\]
Take now  $\bar\varepsilon>0$ sufficiently small so that $\lambda_0+\bar\varepsilon<\lambda_1(\nu)$ and,
 for any $0<\varepsilon\leq \bar\varepsilon$
\begin{itemize}
\item[$a)$] $w_{\lambda_0+\varepsilon}^\nu-w\geq \sigma>0$ in $\mathcal{K}\,$,
\item[$b)$] $\mathcal{L}(\Omega_{\lambda_0+\varepsilon}^\nu\setminus \mathcal{K})\leq \delta\,$.
\end{itemize}
Taking into account $a)$ it is now easy to check  that, for any $0<\varepsilon\leq \bar\varepsilon$, we have that $w\leq w_{\lambda_0+\varepsilon}^\nu$ on the boundary of $\Omega_{\lambda_0+\varepsilon}^\nu\setminus \mathcal{K}$. Consequently, by $b)$, we can apply the Weak Comparison Principle  (Proposition \ref{pr:wcp}) and deduce that
\[
w\leq w_{\lambda_0+\varepsilon}^\nu\qquad \text{in} \qquad\Omega_{\lambda_0+\varepsilon}^\nu\setminus \mathcal{K}\,.
\]
Thus  $w\leq w_{\lambda_0+\varepsilon}^\nu$ in $\Omega_{\lambda_0+\varepsilon}^\nu$\, and by applying the Strong Comparison Principle
(Lemma \ref{positivitycomp})\, we have
$w< w_{\lambda_0+\varepsilon}^\nu$ in $\Omega_{\lambda_0+\varepsilon}^\nu$ . We get a contradiction with the definition of $\lambda_0$ and conclude that actually $\lambda_0=\lambda_1(\nu)$. Then \eqref{zxclu:1.8222} is proved.\\

\noindent It follows now directly from simple geometric considerations and by \eqref{zxclu:1.8222} that
$w$ is monotone non-decreasing in $\Omega_{\lambda_1(\nu)}^\nu$ in the $\nu-$direction. This gives
$$\frac{\partial w}{\partial \nu}(x) \geq 0\quad \text{in}\quad\Omega _{\lambda_1(\nu)}^\nu\,,$$ so it is standard to deduce  \eqref{zxclu:1.9222u} from Proposition \ref{pr:wcpmax}.\\
\noindent To prove that
 \eqref{zxclu:1.8222} and \eqref{zxclu:1.9222u} hold true replacing $w$ with $u$, just recall that
 \[
 u=u_0+w\,,
 \]
 and exploit  Proposition \ref{prosimmetri2000}.
\end{proof}

\section{Proof of Theorem \ref{t3}}\label{simmetria2}

The proof of Theorem \ref{t3} is now a direct consequence of Proposition \ref{prosimmetri2}. Observing that by assumption
\[
\lambda _1(\nu)=0\,,
\]
 we can apply Proposition \ref{prosimmetri2} in the $\nu-$direction to get
\begin{equation}\nonumber
u(x)\leq u_{\lambda_1(\nu)}^\nu(x),\qquad \forall x \in \Omega _{0}^\nu.
\end{equation}
and in the $(-\nu)-$direction to get
\begin{equation}\nonumber
u(x)\geq u_{\lambda_1(\nu)}^\nu(x),\qquad \forall x \in \Omega _{0}^\nu.
\end{equation}
Therefore $u(x)\equiv u_{\lambda_1(\nu)}^\nu(x)$ in  $\Omega$. The monotonicity of $u$ follows by \eqref{zxclu:1.9222u}.
%
%
%
%

\bigskip

\end{document}